\documentclass{article}

\usepackage[utf8]{inputenc}
\usepackage{geometry}
\usepackage{amsthm}
\usepackage{amssymb}
\usepackage{amsmath}
\usepackage{graphicx}
\usepackage{float}
\usepackage{tikz}
\usepackage{tkz-berge}
\usepackage{mathtools}
\usepackage[mathscr]{euscript}

\theoremstyle{theorem}
\newtheorem{theorem}{Theorem}[section]
\newtheorem{lemma}[theorem]{Lemma}
\newtheorem{corollary}[theorem]{Corollary}

\theoremstyle{definition}
\newtheorem{definition}[theorem]{Definition}

\newcommand{\lapl}{L}
\newcommand{\kirchhoff}{Kf}
\newcommand{\kemeny}{\mathscr{K}}
\newcommand{\hidden}[1]{}

\title{Resistance distance, Kirchhoff index, and Kemeny's constant in flower graphs}
\author{Nolan Faught\footnote{Department of Mathematics, Brigham Young University, Provo UT, USA, faught3@gmail.com}, Mark Kempton\footnote{Depatrment of Mathematics, Brigham Young University, Provo UT, USA, mkempton@mathematics.byu.edu}, and Adam Knudson\footnote{Department of Mathematics, Brigham Young University, Provo UT, USA, adamarstk@yahoo.com}}
\date{}

\begin{document}

\maketitle

\begin{abstract}
    We obtain a general formula for the resistance distance (or effective resistance) between any pair of nodes in a general family of graphs which we call flower graphs.  Flower graphs are obtained from identifying nodes of multiple copies of a given base graph in a cyclic way.  We apply our general formula to two specific families of flower graphs, where the base graph is either a complete graph or a cycle.  We also obtain bounds on the Kirchhoff index and Kemeny's constant of general flower graphs using our formula for resistance.  For flower graphs whose base graph is a complete graph or a cycle, we obtain exact, closed form expressions for the Kirchhoff index and Kemeny's constant. 
\end{abstract}

\section{Introduction}
    The resistance distance (also called effective resistance) is a tool motivated by ideas from electrical network theory and applications in chemistry that has proven valuable in the study of graphs.  The resistance between two vertices of a graph is defined as follows (see \cite{GraphsAndMatrices}, for example).
    \begin{definition}
        Let $G$ be a connected graph with vertex set $V(G) = \{1,\hdots,n\}$ and let $\lapl$ denote the Laplacian matrix of $G$. The \emph{effective resistance} or \emph{resistance distance} between two vertices $i, j$ is 
        \[r_G(i, j) = (e_i - e_j)^T \lapl^\dagger (e_i - e_j),\]
        where $e_i$ denotes the standard unit vector with a $1$ in the $i$th position and $0$ elsewhere, and $\lapl^\dagger$ represents the Moore-Penrose pseudoinverse of the Laplacian matrix. The \emph{resistance matrix} of $G$ is the matrix whose $i-jth$ entry is $r_G(i, j)$.
    \end{definition}
    
    The resistance distance defines a metric on a graph, and thus gives geometrical insight into graph structure.  The resistance distance has, for example, been applied in graph theory to the areas of link prediction \cite{barrett2019resistance,pachev2018fast} and graph sparsification \cite{SpielSparse}. Resistance distance also has deep connections to the study of random walks on graphs \cite{doylesnell,kirkland2016kemeny}.  A growing literature in graph theory addresses methods for computing the resistance distance in graphs and computing the resistance distance in various families of graphs; see for instance \cite{bapatdvi,barrett2019resistance,2Forests,essam2014resistance,lostinspace,flower,zhang2007resistance} among others.
    
    Resistance distance is closely related to two important constants in graph theory: the Kirchhoff index of a graph, and Kemeny's constant of a graph.  The Kirchhoff index is a measure of the total resistance in a graph, and is an important quantity in electrical network theory that has been widely studied (for instance, see \cite{palacios2001closed,palacios2015kirchhoff,peng2017kirchhoff,yang2008unicyclic, zhang2007resistance} and references therein).  Kemeny's constant is a parameter associated to a random walk on a graph that gives a measure of the average time a random walk takes to reach a vertex \cite{kirkland2012group}.  Kemeny's constant also gives a measure of how well connected a graph is \cite{breen2019computing}.  From work in \cite{kirkland2016kemeny}, Kemeny's constant of a graph can be computed directly if all resistances in the graph are known (see Theorem \ref{thm:kemeny} below).
    
    Recent research in \cite{2Forests} gives a formula that expresses the resistance distance between vertices on a graph with a 2-separation (two vertices whose removal disconnects the graph) in terms of resistances in the subgraphs involved in the 2-separation.  In this paper, we make use of these results to derive an explicit formula for the resistance distance in a general family of graphs which we refer to as flower graphs (see Theorem \ref{thm:mostgen} below).  Given any base graph $G$, the $n$th flower graph of $G$ is the graph obtained by taking $n$ copies of $G$ and identifying a selected pair of vertices in each copy in a cyclic nature.  See Figure \ref{fig:general_flower} for an illustration.  The precise definition is in Definition \ref{def:flower}.  
    \begin{figure}[H]\label{fig:general_flower}
        \centering
        \begin{center}
        \begin{tikzpicture}
        \tikzstyle{every node}=[circle, fill=black, inner sep=1pt]
        \draw{ \foreach \x in {0,60,120,180,240,300} {
        (\x:1)node{}
        }
        };
        \foreach \x in {30,90,150,210,270,330}{
        \draw[rotate around={\x: (\x:1)},](\x:1) ellipse(10 pt and 15 pt);
        }
        \tikzstyle{every node}=[circle, draw=none, fill=white, minimum width = 8 pt, inner sep=1pt]
        \draw \foreach \x in {30,90,150,210,270,330}{
        (\x:1)node[]{$G$}
        };
        \end{tikzpicture}{}
        \end{center}
        \caption{The $6$th flower graph of $G$}
        \label{fig:flowergraphG}
    \end{figure}
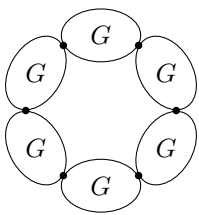
    
    With our explicit formula for resistance, we are able to show that the maximum resistance in a flower graph becomes unbounded as $n$ approaches infinity.  We are also able to bound the Kirchhoff index and Kemeny's constant in general flower graphs.  In addition, we apply our results to some specific families of flower graphs, namely those where the base graph is a complete graph, and those where the base graph is a simple cycle. This yields very simple formulas for the resistance in these specific families of flower graphs.  Using these, we are further able to compute exact formulas for the Kirchhoff index and Kemeny's constant for these graphs.  
    
    We remark that the family of flower graphs we have defined here can be viewed as a generalization of the family of graphs referred to as $(x,y)$-flower graphs in \cite{flower}, in which the resistance of those graphs is obtained.  Our general construction also contains as an example the Sierpinski triangle graphs, whose resistance is determined in \cite{2Forests,jiang2018some}.  Some families of flower graphs also appear in the family of graphs whose resistance and Kirchhoff index are considered in \cite{Yin2018resistance}.  
    
    In work in \cite{lostinspace}, the resistance distance in random geometric graphs is analyzed, and it is shown that as the number of vertices in a random geometric graph grows to infinity, the resistance distance between two nodes approaches the sum of the reciprocals of their degrees.  The authors of \cite{lostinspace} thus argue that the resistance distance is not meaningful as a metric in random geometric graphs since the limiting resistance remains bounded and depends only on degrees, and not the structure of the network.  The results of the current paper are in sharp contrast to this paradigm, since the resistance becomes unbounded as the flower graph grows for any choice of base graph (see Theorem \ref{thm:mostgendiff} and Corollary \ref{cor:mostgenliminf}).  Indeed, our results add to a growing body of research exhibiting families of graphs with this property.  See \cite{barrett2019resistance}, for instance, for a discussion of this issue. Interestingly, many flower graphs (depending on the base graph chosen) can be viewed as ``geometric" graphs, in that they can be exhibited as points in the plane which are adjacent if they are within a certain distance of each other, but they are not \emph{random} geometric graphs as considered in \cite{lostinspace}.  It is of interest to determine generally when the resistance distance in a family of graphs will behave  more like random geometric graphs of \cite{lostinspace}, or more like graphs we are considering here. 
    
\section{Preliminaries}
    \subsection{N-separations of Graphs}
        Our methods for deriving explicit formulas for resistance distance rely heavily on creating $n$-separations of graphs (defined below)
        with easy to compute effective resistances.
        
        \begin{definition}
            An \emph{$n$-separation} on a graph $G$ is a pair of subgraphs $G_1, G_2$ such that
            \begin{itemize}
                \item $V(G) = V(G_1) \cup V(G_2)$,
                \item $|V(G_1) \cap V(G_2)| = n$,
                \item $E(G) = E(G_1) \cup E(G_2)$, and
                \item $E(G_1) \cap E(G_2) = \emptyset$        
            \end{itemize}
            The set $V(G_1) \cap V(G_2) = \{v_1, \cdots, v_n\}$ is called an \emph{$n$-separator} of $G$.
        \end{definition}
        
        \begin{lemma}[Equation 4 of \cite{2Forests}]\label{lem:onesep}
            Given a graph $G$ with a 1-separator $u \in V(G)$, let $G_1$ and $G_2$ represent the two graphs created by the 1-separation. \\
            If $i \in V(G_1)$ and $j \in V(G_2)$,
            \begin{equation}
                r_G(i,j) = r_{G_1}(i, u) + r_{G_2}(j,u)\label{eq:onesep}
            \end{equation}
        \end{lemma}
    
        \begin{lemma}[Theorem 18 of \cite{2Forests}]\label{lem:twosep}
            Let $G$ be a graph with a 2-separation, with $i,j$ the two vertices separating the graph, and $G_1, G_2$ the two graphs created by the separation. \\
            If $u,v$ are in the vertex set of $G_1$, then
            \begin{equation}
                r_G(u,v) = r_{G_1}(u,v) - 
                           \frac{[r_{G_1}(u,i) + r_{G_1}(v,j) - r_{G_1}(u,j) - r_{G_1}(v,i)]^2}{4[r_{G_1}(i,j) + r_{G_2}(i,j)]}\label{eq:twosep}
            \end{equation}
        \end{lemma}

    \subsection{Kirchhoff Index and Kemeny's Constant}
        \begin{definition}\label{def:kirchhoff}
            Given a graph $G$, the \emph{Kirchhoff index} $\kirchhoff(G)$ is given by the summation
            \begin{align*}
                \kirchhoff(G) = \frac{1}{2}\sum_{i,j \in G} r_G(i,j).
            \end{align*}
        \end{definition}
        
        Kemeny's constant is a quantity arising in the study of Markov chains, which is described in more detail in \cite{kirkland2016kemeny} (for example). For a random walk on a graph, Kemeny's constant gives a measure of the average length of a random walk between two vertices of the graph. We will not need the full definition of Kemeny's constant here, but we will use the following result from \cite{kirkland2016kemeny} to calculate Kemeny's constant in terms of resistance.
        \begin{theorem}[Corollary 2.4 of \cite{kirkland2016kemeny}]\label{thm:kemeny}
            Let $R$ be the resistance matrix of a connected graph $G$ with $n$ vertices (i.e., the matrix whose $(i,j)$ entry is $r_G(i,j)$), $q$ be the number of edges in $G$, and $d$ be the vector of degrees $d_1, d_2, ..., d_n$. Kemeny's constant is given by
            \begin{equation*}
                \kemeny(G) = \frac{1}{4q}d^TRd = \frac{1}{4q}\sum_{i, j \in G} d_i d_j r_G(i, j).
            \end{equation*}
        \end{theorem}

\section{Generalized Flower Graphs}
    We begin with the most general result, which is the main result of this paper. First, we define the class of graphs that we are working with and then proceed to give explicit formulas for resistance distance
    in terms of the effective resistance in smaller subgraphs.
    
    \begin{definition}\label{def:flower}
        Let $G$ be a graph, $x, y$ be two distinct vertices of $G$, and $n \geq 3$. A \emph{generalized flower} of $G$, written $F_n(G, x, y)$, is the graph obtained by taking $n$ vertex disjoint copies of the base graph $G_1, G_2, \cdots G_n$, and associating $x_{i-1}$, the marked vertex $x$ in $G_{i-1}$, with $y_{i}$ for $1 < i < n$ and $x_1$ with $y_n$. We refer to $G_i$ as the $i$-th \emph{petal} of the flower graph, and the set $I = \{x_1, \cdots, x_n\}$ as the \emph{associated vertices} of the flower.
    \end{definition}
    
    We suppress the marked vertices $x, y$ from our notation when their choice is clear from context or the specification is unnecessary.

    The following theorem is our main result, which expresses the resistance in any flower graph $F_n(G)$ in terms of resistances in the base graph $G$.
    
    \begin{theorem}\label{thm:mostgen}
    Given a generalized flower graph $F_n(G)=F_n(G,x,y)$ with vertices $u, v$ in different copies of $G$, label the copies such that $u \in V(G_1)$. Let $d$ be the number of copies of $G$ between $u,v$ inclusive, that is, $v \in V(G_d)$. Let $x,y$ be the vertices of $G$ connecting each $G_i$ with $G_{i+1}$ and $G_{i-1}$. Then we have
    
    \begin{align*}
        r_{F_n(G)}(u,v) =&\: r_G(u,y) + r_G(v,x) +(d-2)r_G(x,y) \\
                         &- \frac{[r_G(u,x)+r_G(v,y)-r_G(u,y)-r_G(v,x)-2(d-1)r_G(x,y)]^2}{4nr_G(x,y)}.
    \end{align*}
    If $u,v$ are both in the same copy of $G$,
    \begin{align*}
        r_{F_n(G)}(u,v) = r_G(u,v) - \frac{[r_G(u,x)+r_G(v,y)-r_G(u,y)-r_G(v,x)]^2}{4nr_G(x,y)}.
    \end{align*}

    \end{theorem}

    \begin{proof}
    We first prove the formula when $u,v$ are in different copies of $G$. Label such that $u\in V(G_1)$ and $v\in V(G_d)$. If we let $\{x_1, y_d\}$ be a 2-separator on $F_n(G)$, we have a 2-separation such that $u$ and $v$ are in the same component. We sometimes refer to $\{x_1, y_d\}$ as $\{i,j\}$ as in Lemma \ref{lem:twosep}.
    Let $H_1$ be the graph of the separation containing $u, v$ and $H_2$ be the rest of the flower graph (see Figure \ref{fig:2sepbreakdown1}). Then by Lemma \ref{lem:onesep}
    \begin{align*}
        r_{H_1}(u,v) = r_G(u,y) + (d-2)r_G(x,y) + r_G(x,v).
    \end{align*}{}
    Due to our labeling we also have
    \begin{align*}
        r_{H_1}(u,i) &= r_G(u,x) \text{ and} \\
        r_{H_1}(v,j) &= r_G(v,y)
    \end{align*}{}
    Once again using Lemma \ref{lem:onesep} we get
    \begin{align*}
        r_{H_1}(u,j) &= r_G(u,y) + (d-1)r_G(x,y) \\
        r_{H_1}(v,i) &= r_G(v,x) + (d-1)r_G(x,y) \\
        r_{H_1}(i,j) &= d\cdot r_G(x,y) \text{ and} \\
        r_{H_2}(i,j) &= (n-d)r_G(x,y)
    \end{align*}
    
    \begin{figure}[H]\label{fig:2sepbreakdown1}
        \centering
        \begin{center}
        \begin{tikzpicture}
        \tikzstyle{every node}=[circle, fill=black, inner sep=1pt]
        \draw{ \foreach \x in {0,60,120,180,240,300} {
        (\x:1)node{}
        }
        };
        \draw{(0:1)node[label=$j$,label distance=3pt]{}};
        \draw{(180:1)node[label=$i$,label distance=3pt]{}};
        \foreach \x in {30,90,150,210,270,330}{
        \draw[rotate around={\x: (\x:1)},](\x:1) ellipse(10 pt and 15 pt);
        }
        \tikzstyle{every node}=[circle, draw=none, fill=white, minimum width = 8 pt, inner sep=1pt]
        \draw \foreach \x in {30,90,150,210,270,330}{
        (\x:1)node[]{$G$}
        };
        \tikzstyle{every node}=[circle, fill=black, inner sep=1pt]
        \draw{(150:1.35)node[label=$u$]{}};
        \draw{(30:1.35)node[label=$v$]{}
        };
        \end{tikzpicture}{}
        \end{center}
        \caption{$F_6(G)$ with the $i,j$ 2-separation and nodes $u,v$ labeled}
        \label{fig:flowergraphGv2}
    \end{figure}
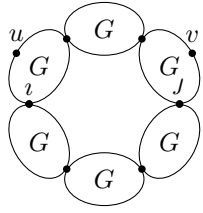
    
    \begin{figure}[H]
        \centering
        \begin{tikzpicture}
        \tikzstyle{every node}=[circle, fill=black, inner sep=1pt]
        \foreach \x in {1,2,3,4}{\draw { 
        (\x,1)node{}
        };
        }
        \foreach \x in{1,2,3}{
        \draw[] (\x+0.5,1)ellipse(14pt and 10pt);
        }
        \foreach \x in {5,6,7}{
        \draw{(\x,1)node{}};
        \draw[] (\x+0.5,1)ellipse(14pt and 10pt);
        }
        \draw{(8,1)node{}};
        
        \tikzstyle{every node}=[circle, draw=none, fill=white, minimum width = 6 pt, inner sep=1pt]
        \draw \foreach \x in {1,2,3,5,6,7}{
        (\x+0.5,1)node[]{$G$}
        };
        \tikzstyle{every node}=[circle, fill=black, inner sep=1pt]
        \draw{(1,1)node[label={$i$}]{}};
        \draw{(4,1)node[label=$j$]{}};
        \draw{(5,1)node[label=$j$]{}};
        \draw{(8,1)node[label={$i$}]{}};
        \draw{(1.5,1.35)node[label=$u$]{}};
        \draw{(3.5,1.35)node[label=$v$]{}};
        \end{tikzpicture}
        \caption{$F_6(G)$ after applying the 2-separation}
        \label{fig:2sepbreakdown2}
    \end{figure}
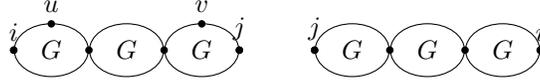{}    
    Plugging these values into Lemma \ref{lem:twosep} we get
    \begin{align*}
        r_{F_n(G)}(u,v) =& \: r_G(u,y) + r_G(v,x) + (d-2)r_G(x,y)\\&-\frac{[r_G(u,x)+r_G(v,y)-r_G(u,y)-(d-1)r_G(x,y)-r_G(v,x)-(d-1)r_G(x,y)]^2}{4[dr_G(x,y)+(n-d)r_G(x,y)]}\\
        &=r_G(u,y) + r_G(v,x) + (d-2)r_G(x,y)\\
        & -\frac{[r_G(u,x)+r_G(v,y)-r_G(u,y)-r_G(v,x)-2(d-1)r_G(x,y)]^2}{4nr_G(x,y)}
    \end{align*}{}
    Thus we have arrived at our desired result.
    
    Now we look at when $u,v$ are in the same copy of $G$. This is really just a special case of Lemma \ref{lem:twosep}. Note that as above we get
    \begin{align*}
        r_{H_1}(i,j) + r_{H_2}(i,j) = nr_G(x,y).
    \end{align*}
    \hidden{
    Now we look at when $u,v$ are in the same copy of $G$. This is really just a special case of Lemma \ref{lem:twosep}. The vertices $i,j$ from Lemma \ref{lem:twosep} become vertices $x,y$, and,  similar to the above proof we get
    \begin{align*}
        r_{H_1}(x,y) + r_{H_2}(x,y) = nr_G(x,y).
    \end{align*} 
    }
    \end{proof}
    
    
    
    The next theorem will address where one might find the maximum effective resistance in a flower graph. This class of graphs contains many symmetries, which causes the maximum effective resistance to occur at several points. Bapat shows that resistance distance satisfies the properties of a metric on a graph.  In particular, it satisfies the triangle inequality, so resistance distance also satisfies the following reverse triangle inequality (see Chapter 10 of \cite{GraphsAndMatrices}).
    
    \begin{lemma}\label{lem:ReverseTriangleIneq}
        Let $G$ be any graph, and let $x, y, z$ be any vertices of $G$. Then
        \begin{align*}
            \lvert r_G(x,y) - r_G(y,z) \rvert \leq  r_G(x,z)
        \end{align*}
    \end{lemma}
    
    
    \begin{theorem}\label{thm:maxd}
    Let $F_n(G)$ be as defined above. The effective resistance between two vertices $u, v \in F_n(G)$ will be greatest when $u \in G_1$ and $v \in G_d$ where $d$ is between $d = \frac{n}{2}$  and $d = \frac{n}{2}+2$. If $n$ is odd, then the maximum will always occur at $d = \frac{n+1}{2}$.
    \end{theorem}
    \begin{proof}
    Treating Theorem \ref{thm:mostgen} as a function of $d$, standard techniques show this function is increasing for $d \leq\frac{n+2}{2} + \frac{r_G(u,x)+r_G(v,y)-r_G(u,y)-r_G(v,x)}{2r_G(x,y)}$, decreasing for $d \geq \frac{n+2}{2} + \frac{r_G(u,x)+r_G(v,y)-r_G(u,y)-r_G(v,x)}{2r_G(x,y)}$. By Lemma \ref{lem:ReverseTriangleIneq}, the expression $\frac{r_G(u,x)+r_G(v,y)-r_G(u,y)-r_G(v,x)}{2r_G(x,y)}$ is between $-1$ and $1$, and thus this function will achieve its maximum at some $d$ such that $\frac{n}{2}\leq d \leq \frac{n}{2}+2$.  
    \end{proof}
    One might have expected the maximum resistance in a flower graph to always occur between copies of $G$ that are as far apart as possible, or in other words at $d=\frac{n}{2}+1$, but this result suggests otherwise. Below is an example of a flower graph where the maximum resistance distance can occur at one of these less expected values of $d$.
    
    \begin{figure}[H]
    \centering
    \begin{tikzpicture}[scale=0.5]
    \tikzstyle{every node}=[circle, fill=black, inner sep=0.75pt]
    \draw{
    (180-26.57:2.24)node{}--(180-63.43:2.24)node{}
    (180-26.57:2.24)node{}--(180:1)node{}
    (90:1)node{}--(180-63.43:2.24)node{}
    (90:1)node[label=above:{$u$}]{}--(180:1)node{}
    (180-26.57:2.24)node{}--(180-45:1.414)node{}
    (180-45:1.414)node{}--(90:1)node{}
    (90:1)node{}--(45:1.414)node{}
    (45:1.414)node{}--(26.57:2.24)node{}
    (26.57:2.24)node{}--(18.43:3.16)node{}
    (18.43:3.16)node[label=above:{$v$}]{}--(14.04:4.12)node{}
    };
\end{tikzpicture}
    \qquad
    \begin{tikzpicture}[scale = 0.25]
    \tikzstyle{every node}=[circle, fill=black, inner sep=0.75pt]
    \draw{ (2,0)node[label=above:{$u$}]{}
    (6,-5)node[label=right:{$v$}]{}};
    \draw{\foreach \x in {0,1,2,3,4,5}{
    (\x,0)node{}--(\x+1,0)node{}
    }};
    \draw{\foreach \x in {0,2}{
    (\x,0)node{}--(1,0.5)node{}
    (\x,0)node{}--(1,-0.5)node{}
    }};
    \draw{\foreach \x in {0,1,2,3,4,5}{
    (\x,-6)node{}--(\x+1,-6)node{}
    }
    (1,-6)node[label=below:{$w$}]{}};
    \draw{\foreach \x in {4,6}{
    (\x,-6)node{}--(5,-5.5)node{}
    (\x,-6)node{}--(5,-6.5)node{}
    }};
    \draw{\foreach \y in {-5,-4,-3,-2,-1,0}{
    (0,\y)node{}--(0,\y-1)node{}
    }};
    \draw{\foreach \x in {-6,-4}{
    (0,\x)node{}--(0.5,-5)node{}
    (0,\x)node{}--(-0.5,-5)node{}
    }};
    \draw{\foreach \y in {-5,-4,-3,-2,-1,0}{
    (6,\y)node{}--(6,\y-1)node{}
    }};
    \draw{\foreach \x in {0,-2}{
    (6,\x)node{}--(6.5,-1)node{}
    (6,\x)node{}--(5.5,-1)node{}
    }};
\end{tikzpicture}
    \qquad
    \begin{tikzpicture}[scale=0.2]
    \tikzstyle{every node}=[circle, fill=black, inner sep=0.75pt]
    \draw{
    (0,10)node{}--(1,9.27)node{}--(2,8.548)node{}--(3,7.822)node{}--(4,7.096)node{}--(5,6.37)node{}--(6,5.644)node{}
    (0,10)node{}--(1.5,9.96)node{}--(2,8.548)node{}--(0.5,8.585)node{}--(0,10)node{}
    (0,10)node{}--(-1,9.27)node{}--(-2,8.548)node{}--(-3,7.822)node{}--(-4,7.096)node{}--(-5,6.37)node{}--(-6,5.644)node{}
    (-6,5.644)node{}--(-5.5,7.059)node{}--(-4,7.096)node{}
    (-6,5.644)node{}--(-4.5,5.681)node{}--(-4,7.096)node[label=above :{$u$}]{}
    (6,5.644)node{}--(5.618,4.665)node{}--(5.236,3.686)node{}--(4.854,2.707)node{}--(4.4715,1.727)node{}--(4.089,0.748)node[label= right:{$v$}]{}--(3.70725,-0.23)node{}
    (6,5.644)node{}--(6.2,4.437)node{}--(5.236,3.686)node{}        (6,5.644)node{}--(5,4.906)node{}--(5.236,3.686)node{}
    (-6,5.644)node{}--(-5.618,4.665)node{}--(-5.236,3.686)node{}--(-4.854,2.707)node{}--(-4.4715,1.727)node{}--(-4.089,0.748)node{}--(-3.70725,-0.23)node{}
    (-3.70725,-0.23)node{}--(-4.8,0.47)node{}--(-4.4715,1.727)node{}        (-3.70725,-0.23)node{}--(-3.3,1.056)node{}--(-4.4715,1.727)node{}
    (-3.70725,-0.23)node{}--(-2.4715,-0.23)node{}--(-1.236,-0.23)node{}--(0,-0.23)node{}--(1.235,-0.23)node{}--(2.4715,-0.23)node{}--(3.70725,-0.23)node{}
    (3.70725,-0.23)node{}--(2.4715,0.5)node{}--(1.235,-0.23)node{}
    (3.70725,-0.23)node{}--(2.4715,-1)node{}--(1.235,-0.23)node{}
    };
\end{tikzpicture}
    \caption{The base graph $G$ (left) $F_4(G)$ (center) $F_5(G)$ (right)}
    \label{fig:maxcounterexample}
\end{figure}
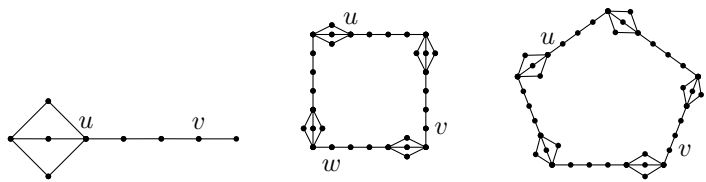{}
    In Figure \ref{fig:maxcounterexample} the maximum resistance distance between copies of vertices $u,v$ from $G$ will occur at points $u,v\in F_4(G)$ where $d = 2$ as opposed to $u,w \in F_4(G)$ where $d = 3$. For $F_5(G)$ the max for those specific vertices occurs where one would expect.
    
    \begin{theorem}\label{thm:mostgendiff}
        Let $F_n(G)$ and $F_{n+1}(G)$ be generalized flower graphs as defined above. Let $u, v$ be vertices with the largest effective resistance distance in the graph. That is, $u \in G_1$ and $v \in G_d$ where $ \frac{n}{2} \leq d \leq \frac{n}{2}+2$. Then
        \begin{align*}
            \lim_{n\to\infty}[ r_{F_{n+1}(G)}(u, v) - r_{F_n(G)}(u, v) ]= \frac{1}{4}r_G(x,y)
        \end{align*}
    \end{theorem}
   
    \begin{proof}
    Assume $r_{F_n(G)}(u, v)$ is a maximum for $F_n(G)$. Then $d = \frac{n}{2}+ \alpha$ where $0 \leq \alpha \leq 2$. Assume similarly that $r_{F_{n+1}(G)}(u, v)$ is a maximum for $F_{n+1}(G)$. Then $d = \frac{n+1}{2}+ \beta$ where $0 \leq \beta \leq 2$. Then by Theorem \ref{thm:mostgen} we have
    \begin{align*}
        r_{F_{n+1}(G)}(u,v) =&\: r_G(u,y)+r_G(v,x)+\left(\frac{n+1}{2}+\beta - 2\right)r_G(x,y) \\
        &- \frac{[r_G(u,x)+r_G(v,y)-r_G(u,y)-r_G(v,x)-2(\frac{n+1}{2}+ \beta-1)r_G(x,y)]^2}{4(n+1)r_G(x,y)}
    \end{align*}{}
    
    and also
    
    \begin{align*}
        r_{F_n(G)}(u,v) =&\: r_G(u,y) + r_G(v,x) + \left(\frac{n}{2}+\alpha -2\right)r_G(x,y) \\ &- \frac{[r_G(u,x)+r_G(v,y)-r_G(u,y)-r_G(v,x)-2(\frac{n}{2}+\alpha-1)r_G(x,y)]^2}{4nr_G(x,y)}.
    \end{align*}{}
    
    For convenience in writing, let $\gamma = r_G(u,x)+r_G(v,y)-r_G(u,y)-r_G(v,x)-2\alpha r_G(x,y)+2r_G(x,y)$ and $\lambda = r_G(u,x)+r_G(v,y)-r_G(u,y)-r_G(v,x)-2\beta r_G(x,y)+r_G(x,y)$. 
    Plugging in $\gamma$ and $\lambda$ and subtracting the previous two equations yields
    \begin{align*}
        r_{F{n+1}(G)}(u,v)-r_{F_n(G)}(u,v) =&\: \left(\frac{1}{2}+\beta-\alpha\right)r_G(x,y)+\frac{(n+1)[\gamma - nr_G(x,y)]^2 - n[\lambda - nr_G(x,y)]^2}{4n(n+1)r_G(x,y)} \\
            =&\: \left(\frac{1}{2}+\beta -\alpha\right)r_G(x,y) + (n+1)\frac{\gamma^2-2\gamma n r_G(x,y)+n^2r_G(x,y)^2}{4n^2(1+\frac{1}{n})r_G(x,y)} \\
             &-n\frac{\lambda^2-2\lambda nr_G(x,y)+n^2r_G(x,y)^2}{4n^2(1+\frac{1}{n})r_G(x,y)} \\
            =&\: \left(\frac{1}{2}+\beta-\alpha\right)r_G(x,y)\\
            &+\frac{n^2r_G(x,y)(2\lambda -2\gamma +r_G(x,y))+n(\gamma^2-2\gamma r_G(x,y)-\lambda^2)+\gamma^2}{4n^2(1+\frac{1}{n})r_G(x,y)}
    \end{align*}{}
    Note that $2\lambda - 2\gamma = 4\alpha r_G(x,y) - 4\beta r_G(x,y) -2r_G(x,y)$. Now taking the limit as $n$ goes to infinity we have
    \begin{align*}
        \lim_{n \to \infty} r_{F{n+1}(G)}(u,v)-r_{F_n(G)}(u,v) =&\: \lim_{n \to \infty}\left[ \left(\frac{1}{2}+\beta-\alpha\right)r_G(x,y)+\frac{n^2r_G(x,y)^2(4\alpha -4\beta -1)}{4n^2(1+\frac{1}{n})r_G(x,y)}\right.\\
        &\left.+\frac{n(\gamma^2-2\gamma r_G(x,y)-\lambda^2)+\gamma^2}{4n^2(1+\frac{1}{n})r_G(x,y)}\right] \\
        =&\: \left(\frac 12 +\beta - \alpha\right)r_G(x,y)+\alpha r_G(x,y) - \beta r_G(x,y)-\frac 14 r_G(x,y) \\
        =&\: \frac 14 r_G(x,y).
    \end{align*}{}
    \end{proof}
    
    \begin{corollary}\label{cor:mostgenliminf}
        For a class of flower graphs with the same base graph $G$,
        \begin{align*}
            \lim_{n \to \infty} \max_{u,v}(r_{F_n(G)}(u,v)) = \infty.
        \end{align*}
    \end{corollary}

    \subsection{Bounds for Kirchhoff Index and Kemeny's Constant}
    While we have not derived formulae for the Kirchhoff Index and Kemeny's constant for generalized flower graphs, we have derived bounds on these values.
    
    \begin{theorem}{}\label{thm:kirchbounds}
        Let $\kirchhoff(F_n(G))$ be the Kirchhoff index for the $n$th flower graph of $G$ and $\kirchhoff(G)$ be Kirchhoff index for the base graph $G$. Let $|V(G)| = m$. Then the following inequality holds.
        \begin{align*}
            n\kirchhoff(G)-\frac{m(m-1)r(x,y)}{2} \leq \kirchhoff(F_n(G)) \leq \kirchhoff(G)(n+nm(n-1))+\frac{r_G(x,y)(n^3-n^2)m^2}{4}
        \end{align*}{}
    \end{theorem}
    \begin{proof}
     Here we will write the Kirchhoff Index in terms of the resistances that exist within a copy of $G$ and the resistances that span into different copies of $G$. We refer to resistance distance in $F_n(G)$ as $r_F(i,j)$ and resistance distances in $G$ as $r(i,j)$. \\
     
     For the lower bound we will add only the resistances between vertices that are in the same copy of $G$ by using Theorem \ref{thm:mostgen}. We also make use of Lemma \ref{lem:ReverseTriangleIneq} in the third line.
    \begin{align*}
        \kirchhoff(F_n(G)) =&\:\frac{n}{2}\sum_{i,j\in G_1}r_F(i,j)+\sum_{\substack{i\in G_k, j\in G_l\\i\notin G_l}}r_F(i,j)\\
        \geq&\: \frac{n}{2}\sum_{i,j\in G}\left(r(i,j)-\frac{[r(i,x)+r(j,y)-r(i,y)-r(j,x)]^2}{4nr(x,y)}\right)\\
        \geq&\:n\kirchhoff(G)-\frac{n}{2}\sum_{i,j\in G}\frac{r(x,y)}{n}\\
        =&\:n\kirchhoff(G)-\frac{m(m-1)r(x,y)}{2}
    \end{align*}{}

    Now for the upper bound. We again will add resistances in the same copy of $G$ and those in strictly different copies of $G$ using Theorem \ref{thm:mostgen}. 
    \begin{align*}
        \kirchhoff(F_n(G)) =& \: \frac{1}{2}\sum_{i,j\in F}r_F(i,j)\\
        =&\:\frac{n}{2}\sum_{i,j\in G_1}r_F(i,j) + \smashoperator{\sum_{\substack{i\in G_k,j\in G_l\\j\notin G_k}}}r_F(i,j)\\
        \leq&\: n\kirchhoff(G)+\frac{n}{2}\sum_{d=2}^n\sum_{i=1}^m\sum_{j=1}^m\left(r(i,y)+r(x,j)+(d-2)r(x,y)\right)\\
        =&\:n\kirchhoff(G)+\frac{n(n-1)m}{2}\sum_{i=1}^mr(i,y)+\frac{n(n-1)m}{2}\sum_{j=1}^mr(x,j)+\frac{r(x,y)n^2(n-1)m^2}{4}\\
        \leq&\: n\kirchhoff(G)+n(n-1)m\kirchhoff(G)+\frac{r(x,y)n^2(n-1)m^2}{4}\\
        =&\:\kirchhoff(G)(n+nm(n-1))+\frac{r(x,y)(n^3-n^2)m^2}{4}
    \end{align*}{}
    \end{proof}
    The lower bound on Kirchhoff index is admittedly quite rough as we are throwing away a lot of information in the proof.  However, the Kirchhoff index of a flower graph with $G=P_2$ and $n=3$ will achieve our lower bound. Note that $F_3(P_2)$ is simply a complete graph on 3 vertices. In Sections 3 and 4 we find exact expressions for certain families of flower graphs. These examples suggest that the upper bound is closer to the true value.
    \begin{theorem}\label{thm:kembounds}
    Let $\kemeny(F_n(G))$ be Kemeny's constant for the $n$th flower graph of $G$ and $\kemeny(G)$ be Kemeny's constant for the base graph $G$. Let $|V(G)| = m$ and $|E(G)| = q$. Then the following inequality holds.
    \begin{align*}
        \kemeny(G)-\frac{m(m-1)^3r(x,y)}{2nq}\leq \kemeny(F_n(G)) \leq \kemeny(G)(4n-1) + \frac{r_G(x,y)(n^2-3n+2)(2m-2)^2m^2}{8q_G}
    \end{align*}{}
    \end{theorem}
    \begin{proof}
    We proceed in similar fashion as we did with the Kirchhoff index. Note that $|E(F_n(G))| = nq$. Where necessary we will note that the maximum degree a vertex in a flower graph can obtain is $2(m-1)$.
    \begin{small}
    \begin{align*}
        \kemeny(F_n(G)) =&\:\frac{1}{4nq}\sum_{i,j\in F}d_{i_F}d_{j_F}r_F(i,j)\\
        \geq&\:\frac{1}{4q}\sum_{i,j\in G_1}d_{i_F}d_{j_F}\left(r(i,j)-\frac{[r(i,x)+r(j,y)-r(i,y)-r(j,x)]^2}{4nr(x,y)}\right)\\
        \geq&\:\kemeny(G)-\frac{1}{4}\sum_{i,j \in G_1}\frac{4(m-1)^2r(x,y)}{n}\\
        =&\:\kemeny(G)-\frac{m(m-1)^3r(x,y)}{2nq}
    \end{align*}
    \end{small}
    Now for the upper bound. Since the degree of vertices $x,y$ will be smaller in $G$ than they are in $F_n(G)$ we take caution and account for that in order to preserve the inequality.
    \begin{small}
    \begin{align*}
        \kemeny(F_n(G)) =&\:\frac{1}{4nq}\sum_{i,j\in F}d_{i_F}d_{j_F}r_F(i,j)\\
        =&\:\frac{1}{4q}\sum_{i,j\in G_k}d_{i_F}d_{j_F}r_{G_k}(i,j) + \frac{1}{4nq}\sum_{\substack{i\in G_k,j\in G_l\\j\notin G_k}}d_{i_F}d_{j_F}r_F(i,j)\\
        \leq&\:\kemeny(G)+\frac{1}{4q}\sum_{i\sim y}d_{i_G}d_{x_G}r_G(i,y) + \frac{1}{4q}\sum_{j \sim x}d_{j_G}d_{y_G}r_G(x,j)+\frac{1}{4nq}\sum_{\substack{i\in G_k,j\in G_l\\i\notin G_l}}d_{i_F}d_{j_F}r_F(i,j)\\
        \leq&\:3\kemeny(G)+\frac{1}{4q}\sum_{d=2}^n\sum_{i=1}^m\sum_{j=1}^md_{i_F}d_{j_F}(r_G(i,y)+r_G(x,j)+(d-2)r_G(x,y))\\
        =&\:3\kemeny(G)+\frac{n-1}{4q}\sum_{i=1}^m\sum_{j=1}^md_{i_F}d_{j_F}(r_G(i,y)+r_G(x,j))+\frac{r_G(x,y)(n^2-3n+2)}{8q}\sum_{i=1}^m\sum_{j=1}^md_{i_F}d_{j_F}\\
        \leq&\: 3\kemeny(G)+2(n-1)\kemeny(G)+\frac{n-1}{4q}\sum_{i\sim y}d_{i_G}d_{x_G}r_G(i,y) + \frac{n-1}{4q}\sum_{j \sim x}d_{j_G}d_{y_G}r_G(x,j)\\
        &\quad +\frac{r_G(x,y)(n^2-3n+2)(2m-2)^2m^2}{8q}\\
        \leq&\: (4n-1)\kemeny(G)+\frac{r_G(x,y)(n^2-3n+2)(2m-2)^2m^2}{8q}
    \end{align*}
    \end{small}
    \end{proof}
    Just as the Kirchhoff index lower bound, the lower bound for Kemeny's constant is quite rough. We are unaware of examples achieving the lower bound. In Sections 4 and 5, we derive exact expressions for Kemeny's constant in certain families of flower graphs.  As with the Kirchhoff index, these examples suggest the upper bound is closer to the true value.

\section{Complete Flower Graphs}
    The results given by Theorem \ref{thm:mostgen} are best used by applying them to subclasses of flower
    graphs where the base graph $G$ is from a specific family of graphs. By studying a family
    of graphs in which resistance distance is well-known or easily derived, we are able to derive expressions
    in terms of distances and resistances in the base graph in many cases. If it is possible to derive explicit
    expressions for resistance, it is also possible to create formulae for expressing Kemeny's constant
    and the Kirchoff index explicitly. The first such subclass of flower graphs that we will examine is the complete flower graph.
    
    \begin{definition}\label{def:completeflower}
        A \emph{complete flower graph} is a flower graph where $G = K_m$ for some $m \geq 3$ and $x, y \in G$ are arbitrary provided that $x \neq y$. We denote a complete flower $F_n(K_m)$.
    \end{definition}
    
    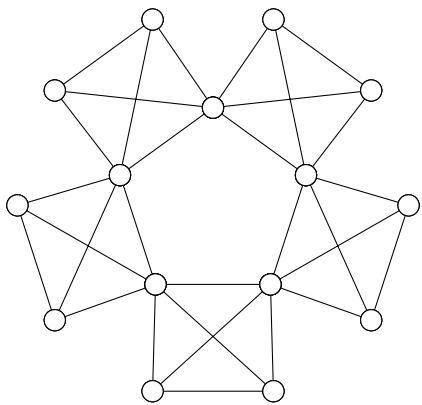
\begin{figure}[H]
        \centering
        \begin{tikzpicture}[scale=1.3]
        \tikzstyle{every node}=[circle, draw, fill=white, minimum width = 8 pt, inner sep=1pt]
            \draw \foreach \x in {18, 90, 162, 234, 306} {
                (\x:1)node{}--(\x+72:1)node{}
                (\x:1)node{}--(\x+18:2)node{}
                (\x+72:1)node{}--(\x+54:2)node{}
                (\x+18:2)node{}--(\x+54:2)node{}
                (\x:1)node{}--(\x+54:2)node{}
                (\x+72:1)node{}--(\x+18:2)node{}
            };
        \end{tikzpicture}
        \caption{$F_5(K_4)$, a complete flower on $5$ copies of $K_4$}
        \label{fig:KFexample}
    \end{figure}
    
    \subsection{Resistance Distance}
        As expressed in the introduction to this section, if we can express resistance distance in the base
        graph simply, the generalized formulae become more useful. We may easily find an expression for the effective resistance on a complete graph. The following Lemma is easily verified with results from chapter 10 of \cite{GraphsAndMatrices}.

        \begin{lemma}\label{lem:complete}
            Let $u, v \in V(K_m)$, where $m \geq 3$. Then the resistance distance between $u$ and $v$ is given by
            \begin{equation}\label{eq:complete}
                r_{K_m}(u, v) = \frac 2{m} \text{ if } u \neq v
            \end{equation}
        \end{lemma}
        \hidden{
            \begin{proof}
                    Let $m \geq 3$ and $u, v \in V(K_m)$. There are two cases \par
                    \emph{Case 1.} $u = v$. Then $r_{K_m}(u, v) = 0$. \par
                    \emph{Case 2.} $u \neq v$. Then there exists an edge between $u$ and $v$ and $m-2$ paths of length
                    $2$ from $u$ to $v$. By the series-parallel rule,
                    \begin{align*}
                        r_{K_m} &= \frac 1{1 + \sum_{i=1}^{m-2} \frac{1}{2}} \\
                                &= \frac 1{\frac{m-2}2 + 1} \\
                                &= \frac 2{m}
                    \end{align*}
            \end{proof}
        }
        
        \begin{theorem}\label{thm:completeflower}
            Let $G$ be a complete flower $F_n(K_m)$ and $u$ and $v$ be vertices in $G$. Recall that I is the set of associated vertices connecting each copy of $K_m$, then
            \begin{align*}
                r_{F_n(K_m)}(u, v) &= \frac {2d(n-d)}{mn} &\text{ if both } u, v \in I \\
                r_{F_n(K_m)}(u, v) &= \frac{2d}m - \frac{(2d-1)^2}{2mn} &\text{ if one of } u, v \in I \\
                r_{F_n(K_m)}(u, v) &= \frac{2d}m - \frac{2(d-1)^2}{mn} &\text{ if neither } u, v \in I
            \end{align*}
            Where $d$ is the number of flower petals separating $u$ and $v$ including the petals containing $u$, and $v$.
        \end{theorem}
        \begin{proof}
            If $u = v$, then $r_{F_n(K_m)}(u, v) = 0$, so we assume that $u \neq v$.\par
            \emph{Case 1.} Suppose that $u, v \in I$. Let $i,j$ be the vertices of the 2-separation as in the proof of Theorem \ref{thm:mostgen}. The simplest 2-separation occurs when we set $u=i$ and $v=j$, so let $G_1$ and $G_2$ be the graphs created by the 2-separator $\{u, v\}$ and $d$ be the number of complete graphs in $G_1$. The terms $r_{G_1}(u,i)$ and $r_{G_1}(v,j)$ are both zero due to the selection of $i$ and $j$, and with a simple summation we have $r_{G_1}(u, v) = r_{G_1}(u, j) = r_{G_1}(v, i) = \sum_{i = 1}^d \frac{2}{m} = \frac{2d}{m}$. From Theorem \ref{thm:mostgen}, we have
            \begin{equation*}
                r_{F_n(K_m)}(u, v) = \frac{2d}{m} - \frac{{\left[-\frac{2d}{m}-\frac{2d}{m}\right]}^2}{4\left(\frac{2n}{m}\right)}
            \end{equation*}
            which gives the desired result when simplified. \par
            \emph{Case 2.} Suppose, without a loss of generality, that $u \in I$ and $v \in O$. We take $i=u$ to be one of the 2-separators and let the other 2-separator $j$ be a vertex adjacent to $v$ such that $u$ and $v$ are in the same component of the 2-separation. The resistances remain identical to those of case 1 with the exception that $r_{G_1}(v, j)$ becomes $\frac{2}{m}$, so
            \begin{align*}
                r_{F_n(K_m)}(u, v) &= \frac{2d}{m} - \frac{{\left[\frac{2}{m} - \frac{2d}{m} - \frac{2d}{m}\right]}^2}{4\left(\frac{2d}{m} + \frac{2(n-d)}{m}\right)} \\
                r_{F_n(K_m)}(x, y) &= \frac{2d}{m} - \frac{(2d-1)^2}{2mn}
            \end{align*} \par
            \emph{Case 3.} Suppose that $u, v \in O$. If we select $i \in I$ to be either vertex adjacent to $u$ and $j \in I$ to be adjacent to $v$ such that $u$ and $v$ are both in the same component of the 2-separation. The only resistance that changes from case 2 is $r_{G_1}(u, i) = \frac{2}{m}$, giving
            \begin{align*}
                r_{F_n(K_m)}(u, v) &= \frac{2d}{m} - \frac{{\left[\frac{2}{m} + \frac{2}{m} - \frac{2d}{m} - \frac{2d}{m}\right]}^2}{4\left(\frac{2d}{m} + \frac{2(n-d)}{m}\right)} \\
                                   &= \frac{2d}{m} - \frac{2{(d-1)}^2}{mn}
            \end{align*}
        \end{proof}

        This gives the interesting result that if $u, v$ are in the same copy of $K_m$ and neither is in $I$,
        $r_{F_n(K_m)}(u, v) = r_{K_m}(u, v)$.
        \begin{theorem}\label{thm:maxcompleteflower}
        The maximum resistance in a complete flower graph $F_n(K_m)$ is given by
        \begin{align*}
            \max(r_{F_n(K_m)}(u,v)) =&\: \frac{n+4}{2m} &\text{if $n$ is even}\\
            \max(r_{F_n(K_m)}(u,v)) =&\: \frac{n^2+4n-1}{2mn} &\text{if $n$ is odd}
        \end{align*}
        \end{theorem}
        \begin{proof}
            Using Theorems \ref{thm:maxd} and \ref{thm:completeflower} to compare potential maximums we find that the largest resistance occurs between nodes $u,v\in O$ with a value of $d=\frac{n}{2}+1$ if $n$ is even and $d=\frac{n+1}{2}$ if $n$ is odd.
        \end{proof}
        
    \subsection{Kirchhoff Index and Kemeny's Constant}
        \begin{theorem}\label{thm:completekirch}
            The Kirchhoff Index of a complete flower is given by
            \begin{align*}
                \kirchhoff(F_n(K_m)) = \frac{n(5+12n+n^2+m^2(-1+6n+n^2)-m(1+18n+2n^2))}{6m}
            \end{align*}{}
        \end{theorem}{}
        \begin{proof}
        Because the closed-form expressions for the resistance distance vary, to compute the Kirchhoff index of a complete flower $F_n(K_m)$, we must take a sum across each of the different expressions.
        
        To get the result we will first add all the resistances between vertices in $I$, this will be our first summation term. Next we add resistances between all possible vertices where exactly one of them is in $I$. That is our second summation term. We next add all the resistances between vertices in the same copy of $K_m$ but are not in $I$. That is our third term. The final summation term adds all possible resistance distances between vertices in different copies of $K_m$ where neither vertex is in $I$.
        \begin{align*}
            \kirchhoff(F_n(K_m)) =&\frac{1}{2}\left( n\sum_{d=1}^{n-1} \left(\frac {2d(n-d)}{mn}\right) +2n(m-2)\sum_{d=1}^n\left(\frac{2d}m - \frac{(2d-1)^2}{2mn}\right)\right.\\
            &+ n(m-2)(m-3)\left(\frac{2}{m}\right)
            \left.+ n(m-2)^2\sum_{d=2}^n\left(\frac{2d}m - \frac{2(d-1)^2}{mn}\right) \right)
        \end{align*}
        Simplifying these summations gives the desired result.
        \end{proof}

        \begin{theorem}\label{thm:completekem}
            Kemeny's Constant of a complete flower is given by
            \begin{align*}
                \kemeny(F_n(K_m)) = \frac{(m-1)(-12n + m(n^2+6n-1))}{6m}
            \end{align*}{}
        \end{theorem}{}
        \begin{proof}
        We begin by noting that there are $\frac{nm(m-1)}{2}$ edges in a complete flower graph. Then we proceed as we did to find the Kirchhoff index only multiplying by the degrees of the vertices as Theorem \ref{thm:kemeny} calls for. Note that if $i\in I$ and $j\in O$ then $d_i=2m-2$ and $d_j=m-1$. Then using Theorem \ref{thm:kemeny} we have
        \begin{align*}
            \kemeny(F_n(K_m)) =&\: \frac{1}{2nm(m-1)}\left(n(2m-2)^2\sum_{d=1}^{n-1} \left(\frac {2d(n-d)}{mn}\right) \right.\\
             &+ 2n(m-2)(m-1)(2m-2)\sum_{d=1}^{n} \left(\frac{2d}m - \frac{(2d-1)^2}{2mn}\right) \\
            &+\left. n(m-2)(m-3)(m-1)^2\frac{2}{m} + n(m-2)^2(m-1)^2\sum_{d=2}^{n} \left(\frac{2d}m - \frac{2(d-1)^2}{mn}\right)\right)
        \end{align*}{}
        Once again, simplifying this expression will yield the desired result.
        \end{proof}
        Comparing these results to the bounds from  Theorems \ref{thm:kirchbounds}, \ref{thm:kembounds} we find that as $n\to\infty$ the ratio of the upper bound for the Kirchhoff index to the actual Kirchhoff index approaches $\frac{3m^2}{(m-1)^2}$. Similarly, we find that as $n\to\infty$ the ratio of the upper bound for the Kemeny's constant to the actual Kemeny's constant approaches $12$.
        
        

    \subsection{Example: $SF_n$}
        \begin{definition}
            A \emph{sunflower graph} is a subclass of flower graphs where $G = K_3$. We denote a sunflower graph with $n$ copies of $K_3$ as $SF_n$. See Figure \ref{fig:SFexample}. \\
        \end{definition}
        
        The construction of $SF_n$ creates a cycle on $n$ vertices consisting of the $u, v$ we selected. We refer to vertices on this cycle as the \emph{inner vertex set} of $SF_n$ and frequently refer to the copies of $K_3$ as the $petals$ of $SF_n$.
    
        \begin{figure}[H]
            \centering
            \begin{tikzpicture}
            \tikzstyle{every node}=[circle, draw, fill=white, minimum width = 8 pt, inner sep=1pt]
                \foreach \x in {30,90,150,210,270,330}{
                \draw{(\x:1)node{}--(\x+60:1)node{}
                (\x-30:1.5)node{}--(\x:1)node{}--(\x+30:1.5)node{}
                };}
            \end{tikzpicture}{}
            \caption{$SF_6$}
            \label{fig:SFexample}
        \end{figure}

        \subsubsection{Formulas for Resistance Distance}
            Due to the previously computed formulae for complete flowers, deriving expressions for the resistance distance
            between vertices on a sunflower graph is trivial.
            \begin{theorem}\label{thm:sun}
                 Let $SF_n$ be a sunflower graph. Recall that I is the set of associated vertices connecting each copy of $K_3$, then
            
                \begin{align}
                    r_{SF_n}(u,v) &= \frac{2d(n-d)}{3n} &\text{ if both } u, v \in I \\
                    r_{SF_n}(u,v) &= \frac{4nd-4d^2+4d-1}{6n} &\text{ if only } u \in I \\
                    r_{SF_n}(u,v) &= \frac{2(nd-(d-1)^2)}{3n} &\text{ if neither } u, v \in I
                \end{align}
                Where d is the number of flower petals separating $u$ and $v$ including the petals containing $u$ and $v$.
            \end{theorem}
            \begin{proof}
                Substituting $m = 3$ into Theorem \ref{thm:completeflower} yields the desired result.
            \end{proof}

            From Theorems \ref{thm:maxcompleteflower}, \ref{thm:mostgendiff} and Corollary \ref{cor:mostgenliminf} we have
            \[\max(r_{SF_n}(u,v)) = \frac{n+4}{6}           \quad\text{ if $n$ is even}\]
            \[\max(r_{SF_n}(u,v)) = \frac{n^2 + 4n - 1}{6n} \quad\text{ if $n$ is odd}\]
            \[\lim_{n\to\infty}[\max(r_{SF_{n+1}}(u,v)-\max(r_{SF_{n}}(u,v)] = \frac{1}{6} \]  \[\lim_{n\to\infty}\max(r_{SF_n}(u,v)) = \infty.\]
            
            
            
        \subsubsection{Kirchhoff Index and Kemeny's Constant}
            \begin{theorem}\label{thm:KirchhoffSF}
            The Kirchhoff Index of a Sunflower Graph on $n$ triangles, $SF_n$, is given by
            \begin{equation*}
                \kirchhoff(SF_{n}) = \frac{1}{18} (4n^3+12n^2-7n).
            \end{equation*}
            \end{theorem}
            \begin{proof}
            Using $m = 3$ with the result in Theorem \ref{thm:completekirch} and simplifying gives the desired result.
            \end{proof}
            
            \begin{theorem}\label{thm:kemenySF}
            Kemeny's constant for a Sunflower Graph $SF_{n}$ is given by
            \begin{equation*}
                \kemeny(SF_{n}) = \frac{1}{3}(n^2+2n-1)
            \end{equation*}
            \end{theorem}
            \begin{proof}
            Plugging in $m=3$ into the formula for Kemeny's constant from Theorem \ref{thm:completekem} and simplifying yields the desired result.
            \end{proof}

\section{Generalized Sunflower Graphs}
    \hidden{
        \begin{figure}[H]
        \centering
        \begin{tikzpicture}[auto, node distance=1.53cm, everyloop/.style={},thick,main node/.style={circle,draw}]
            \node[main node] (1) {1};
            \node[main node] (2) [above right of=1] {2};
            \node[main node] (3) [above right of=2] {3};
            \node[main node] (4) [above right of=3] {4};
            \node[main node] (5) [below right of=4] {5};
            \node[main node] (6) [below right of=5] {6};
            \node[main node] (7) [below right of=6] {7};
            \node[main node] (8) [left of=7] {8};
            \node[main node] (9) [right of=1] {9};
            \node[main node] (10) [above of=1] {10};
            \node[main node] (11) [left of=4] {11};
            
            \path (1) edge node {} (2);
            \path (2) edge node {} (3);
            \path (3) edge node {} (4);
            \path (1) edge node {} (10);
            \path (10) edge node {} (11);
            \path (11) edge node {} (4);
            
            \path (4) edge node {} (5);
            \path (5) edge node {} (6);
            \path (6) edge node {} (7);
            
            \path (7) edge node {} (8);
            \path (8) edge node {} (9);
            \path (9) edge node {} (1);

        \end{tikzpicture}{}
            \caption{}
            \label{fig:SFexample}
    \end{figure}
    }
    
    The construction of \emph{complete flowers} arose from generating a flower graph with the base graph being a complete
    graph. We see sunflower graphs as a subclass of complete flowers, but if we instead take the base graph to be a
    cycle on $n$ vertices, it is possible to construct another class of graphs that contains sunflower graphs.
    
    \begin{definition}
        A \emph{generalized sunflower}, denoted $F_n(C_m)$, is a class of flower graphs obtained by setting $G = C_m$ and
        selecting two vertices $x,y$ in this cycle, then following the construction of flower graphs.
        In each $C_m$ we call the shorter path from $x$ to $y$ $D_1$ and the longer path $D_2$. Let $p=d(x,y)$.
    \end{definition}
    
    \begin{figure}[H]
        \centering
        \begin{tikzpicture}[scale = .85, semithick]
        \tikzstyle{every node}=[circle, draw, fill=white, minimum width = 8 pt, inner sep=1pt]
        \draw {
        (0,0)node{}--(1,0)node{}
        (1,0)node{}--(1,1)node{}
        (1,1)node{}--(1,2)node{}
        (1,2)node{}--(0,2)node{}
        (0,2)node{}--(-1,2)node{}
        (-1,2)node{}--(-1,1)node{}
        (-1,1)node{}--(-1,0)node{}
        (-1,0)node{}--(0,0)node{}
        
        (-1,0)node{}--(-.75,-1)node{}
        (-.75,-1)node{}--(0,-1)node{}
        (0,-1)node{}--(.75,-1)node{}
        (.75,-1)node{}--(1,0)node{}
        (1,0)node{}--(2,.25)node{}
        (2,.25)node{}--(2,1)node{}
        (2,1)node{}--(2,1.75)node{}
        (2,1.75)node{}--(1,2)node{}
        (1,2)node{}--(.75,3)node{}
        (.75,3)node{}--(0,3)node{}
        (0,3)node{}--(-.75,3)node{}
        (-.75,3)node{}--(-1,2)node{}
        (-1,2)node{}--(-2,1.75)node{}
        (-2,1.75)node{}--(-2,1)node{}
        (-2,1)node{}--(-2,.25)node{}
        (-2,.25)node{}--(-1,0)node{}
        };
    \end{tikzpicture}
    \caption{$F_4(C_6)$}
    \end{figure}
    
    \subsection{Resistance Distance}
        \begin{lemma}\label{lem:rescyc}
            Let $C_m$ be a cycle on $m$ vertices. Then the resistance distance between any
            vertices $u,v \in V(C_m)$ is given by the formula:
            \begin{equation}\label{eq:rescyc}
                r_{C_m}(u,v) = \frac{(m-d(u,v)) d(u,v)}{m}
            \end{equation}
            where $d(u, v)$ indicates the standard distance between two vertices of a graph.
        \end{lemma}
        \begin{proof}
            This result is easy to verify with techniques from chapter 10 of \cite{GraphsAndMatrices}.
            \hidden{
                Let $G$ be a cycle on $m$ vertices and let two vertices $u, v$ be given. \\
                Since $u,v$ lie on a closed loop, we may use the series-parallel rule to calculate
                resistance between the two vertices. \\
                Let $d(u,v)$ denote the standard distance from $u$ to $v$. Then we have two paths between $u$
                and $v$, one of length $d(u, v)$ and one of length $m - d(u, v)$. \\
                Applying the series-parallel rule to $G$, we obtain:
                \begin{align*}
                    r_G(u, v) &= \frac{1}{\frac{1}{d(u, v)} + \frac{1}{m - d(u,v)}} \\
                        &= \frac{1}{\frac{m}{(m - d(u ,v)) d(u, v)}} \\
                        &= \frac{(m - d(u, v))d(u, v)}{m}
                \end{align*}
            }
        \end{proof}
        
        \begin{theorem}
            If $G$ is a generalized sunflower graph $F_n(C_m)$ and $u \in C_{m_i}$, $v \in C_{m_j}$, $i \neq j$, then
            \begin{align*}
                r_G(u, v) =&\: \frac{(p_u + l)(m-p_u-l) + k(m-k) + p_u(m-p_u)(d-2)}{m} \\
                           &- \frac{[p_v(m-p_v-2k) + p_u(m-p_u+2l) - 2dp_u(m-p_u)]^2}{4nmp_u(m-p_u)}.
            \end{align*}
            If $u, v \in C_{m_i}$, then
            \begin{align*}
                r_G(u, v) &= \frac{(k-l)(m-k+l)}{m} - \frac{p_u(k-l)^2}{nm(m-p_u)} &\text{ if } u,v\in D_i \\
                r_G(u, v) &= \frac{(k+l)(m-k-l)}{m} - \frac{[p_u^2 + p_u(2l-m) + p_v(m-2k-p_v)]^2}{4nmp_u(m-p_u)} &\text{ if } u\in D_i,  v\in D_j, i\neq j
            \end{align*}
            where $d$ is the number of flower petals separating $u$ and $v$, inclusive, $p_u$ is the length of the path from $x_i$ to $y_i$ that does not contain $u$, $p_v$ is the length of the path from $x_j$ to $y_j$ that does not contain $v$, $l$ is the distance from $x$ to $u$ along the path containing $u$, and $k$ is the distance from $x$ to $v$ along the path containing $v$. If $u$ or $v \in I$, we instead define $p_u$ to be the distance from $x$ to $y$ in the base graph $G$. If $p_u=p_v$ label such that $k\geq l$.
        \end{theorem}
        
        \begin{proof}
        We first consider the case where $u\in C_{m_i}$ and $v\in C_{m_j}$. Label such that $C_{m_i} = C_{m_1}$ and $C_{m_j} = C_{m_d}$. Then by Lemma \ref{lem:rescyc} we have
        \begin{align*}
            r_{C_m}(u,y) =&\: \frac{(p_u+l)(m-p_u-l)}{m}  &r_{C_m}(u,x) =& \frac{l(m-l)}{m}\\
            r_{C_m}(v,y) =& \frac{(p_v+k)(m-p_v-k)}{m}  &r_{C_m}(v,x) =& \frac{k(m-k)}{m}   &r_{C_m}(x,y) =& \frac{p_u(m-p_u)}{m}.
        \end{align*}
        Plugging these values into Theorem \ref{thm:mostgen} and simplifying yields the desired result.
        
        Next consider the cases where $u,v$ are in the same copy of $C_m$. The same resistances from above will hold in these cases, all that is left is to determine $r_{C_m}(u,v)$.
        
        If $u,v\in D_i$, label such that $k\geq l$. Then there is a path of length $k-l$ between $u,v$ so we have $r_{C_m}(u,v) = \frac{(k-l)(m-k+l)}{m}$. Also note that in this case $p_u = p_v$. Using this with Theorem \ref{thm:mostgen} and simplifying we get the desired result.
        
        If $u\in D_i$ and $v\in D_j$ we have a path of length $l+k$ between $u,v$ so we have $r_{C_m}(u,v) = \frac{(l+k)(m-l-k)}{m}$. Using this with Theorem \ref{thm:mostgen} and simplifying we get the desired result.
        \end{proof}

        \begin{figure}[H]
        \centering
        \begin{tikzpicture}[scale = .85, semithick]
        \tikzstyle{every node}=[circle, draw, fill=white, minimum width = 8 pt, inner sep=1pt]
        \draw {
        (0,0)node{}--(1,0)node{}
        (1,0)node{}--(1,1)node{}
        (1,1)node[label=left:{$v$}]{}--(1,2)node{}
        (1,2)node[label=above right:{$x_2$}]{}--(0,2)node{}
        (0,2)node{}--(-1,2)node{}
        (-1,2)node{}--(-1,1)node{}
        (-1,1)node{}--(-1,0)node{}
        (-1,0)node{}--(0,0)node{}
        
        (-1,0)node{}--(-.75,-1)node{}
        (-.75,-1)node{}--(0,-1)node{}
        (0,-1)node{}--(.75,-1)node{}
        (.75,-1)node{}--(1,0)node{}
        (1,0)node{}--(2,.25)node{}
        (2,.25)node{}--(2,1)node{}
        (2,1)node{}--(2,1.75)node{}
        (2,1.75)node{}--(1,2)node{}
        (1,2)node{}--(.75,3)node{}
        (.75,3)node{}--(0,3)node[label=above:{$u$}]{}
        (0,3)node{}--(-.75,3)node{}
        (-.75,3)node{}--(-1,2)node{}
        (-1,2)node[label=above left:{$x_1$}]{}--(-2,1.75)node{}
        (-2,1.75)node{}--(-2,1)node{}
        (-2,1)node{}--(-2,.25)node{}
        (-2,.25)node{}--(-1,0)node{}
        };
    \end{tikzpicture}
    \label{fig:lkvarexplain}
    \caption{Here if we use $d=2$ then $l=2$ as $d(u,x)=3$ measured along the path in $D_2$ and $k=1$ since $d(x,v)=1$ measured along the path in $D_1$.}
    \end{figure}
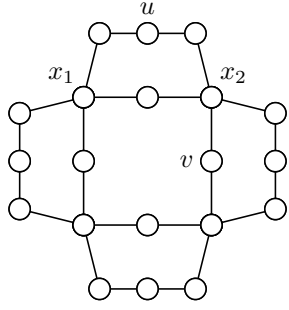
    
    \subsection{Kirchhoff Index and Kemeny's Constant}
        \begin{theorem}
            The Kirchhoff Index of a Generalized Sunflower Graph is given by
            \begin{align*}
                \kirchhoff(F_n(C_m)) =&\: \frac{n(pm-p^2)\left(n^2(m-1)^2+m^2(4-6n)+6mn-1\right)}{12m} \\
                                     &-\frac{n(m^3+m-2-2n(m-1)^2(m+1)}{12}
            \end{align*}
        \end{theorem}
        \begin{proof}
        Our goal is to add the resistance distance over all possible pairs of vertices $u, v$ in our graph. As some of our formulas overlap on a few edge cases we will be careful not to overcount. The first sum below adds all the resistances $r(u,v)$ where $u,v\in D_1$ and $u,v$ are in different copies of $G$. This also catches some of the edge cases where at least one of $u,v$ is a connector vertex. The second sum is over all $u,v\in D_2$ with $u,v$ in different copies of $G$. The third sum adds all the $r(u,v)$ with $u\in D_2, v\in D_1$ with $u,v$ in different copies of $G$. The last three sums will take care of cases where $u,v$ are in the same copy of $G$. The fourth sum adds all the resistances with $u,v\in D_1$. The fifth sum adds resistances with $u\in D_2, v\in D_1$. The last sum adds resistances with $u,v\in D_2$. Notice that except for the first two sums, we must multiply by two in order to add not only $r(u,v)$ but also $r(v,u)$ as Definition \ref{def:kirchhoff} calls for.
        \begin{small}
        \begin{align*}
            \kirchhoff(F_n(C_m)) =& \frac{1}{2}\left[n\sum_{d=2}^n\sum_{l=0}^{p-1}\sum_{k=0}^{p-1} \left(\frac{k(m-k)+(p-l)(m-p+l)+p(m-p)(d-2)}{m}\right.\right.\\
            &\hspace{12em}\left.-\frac{(m-p)(k-l+p(d-1))^2}{nmp}\right)\\
            &+n\sum_{d=2}^n\sum_{l=1}^{m-p-1}\sum_{k=1}^{m-p-1} \left(\frac{(p+l)(m-p-l)+k(m-k)+p(m-p)(d-2)}{m}\right.\\
            &\hspace{12em}\left.-\frac{p(l-k+m+p(d-1)-md)^2}{nm(m-p)}\right) \\
            &+2n\sum_{d=2}^n\sum_{l=1}^{m-p-1}\sum_{k=1}^{p} \left(\frac{(p+l)(m-p-l)+k(m-k)+p(m-p)(d-2)}{m}\right.\\
            &\hspace{12em}\left.-\frac{(p^2(d-1)+p(k+l+m-md)-km)^2}{nmp(m-p)}\right)\\
            &+2n\sum_{l=0}^{p-1}\sum_{k=l+1}^{p-1} \left(\frac{(k-l)(m-k+l)}{m}-\frac{(k-l)^2(m-p)}{nmp}\right)\\
            &+2n\sum_{l=1}^{m-p-1}\sum_{k=1}^{p} \left(\frac{(k+l)(m-l-k)}{m}-\frac{(p(k+l)-km)^2}{nmp(m-p)}\right)\\
            &\left.+2n\sum_{l=1}^{m-p-1}\sum_{k=l+1}^{m-p-1} \left(\frac{(k-l)(m-k+l)}{m}-\frac{p(k-l)^2}{nm(m-p)}\right)\right]
        \end{align*}
        \end{small}
        \hidden{
        \begin{align*}
            \kirchoff(F_n(C_m)) =& \frac{1}{2} \left[n\sum_{v=2}^{n}\sum_{l=0}^{d-1}\sum_{k=1}^{d}\left(\frac{(d-l)(m-d+l)+(d-k)(m-d+k)+(v-2)(m-d)d}{m}\right.\right.\\
            &\hspace{12em}\left.-\frac{(m-d)^2(k+l-dv)^2}{mnd(m-d)}\right)\\
            &+n\sum_{v=2}^{n}\sum_{l=1}^{m-d-1}\sum_{k=1}^{m-d-1}\left(\frac{(m-d-l)(d+l)+(m-d-k)(d+k)+d(m-d)(v-2)}{m}\right.\\
            &\hspace{12em}\left.-\frac{d(k+l+v(d-m))^2}{mn(m-d)}\right)\\
            &+2n\sum_{v=2}^{n}\sum_{l=1}^{m-d-1}\sum_{k=0}^{d-1}\left(\frac{(m-d-l)(d+l)+(d-k)(m+k-d)+(v-2)(m-d)d}{m}\right.\\
            &\hspace{12em}\left.-\frac{(km+d^2v-d(k-l+mv))^2}{mnd(m-d)}\right)\\
            &+2n\sum_{l=0}^{d-1}\sum_{k=1}^{d-l-1}\left(\frac{(d-l-k)(m-d+k+l)}{m}-\frac{(d-k-l)^2(m-d)^2}{mnd(m-d)}\right)\\
            &+2n\sum_{l=1}^{m-d-1}\sum_{k=0}^{d-1}\left(\frac{(l+d-k)(m-l-d+k)}{m}-\frac{(d^2+km-d(m+k-l))^2}{mnd(m-d)}\right)\\
            &+\left.2n\sum_{l=1}^{m-d-1}\sum_{k=1}^{m-d-l-1}\left(\frac{(d+l+k)(m-d-l-k)}{m}-\frac{d(m-d-k-l)^2}{mn(m-d)}\right) \right]
        \end{align*}{}
        }
        Simplifying these sums will yield the desired result.
        \end{proof}

    \begin{theorem}
        Kemeny's Constant is given by
        \begin{align*}
            \kemeny(SF_n(C_m)) =&\:\frac{(n^2-6n+4)(pm-p^2)+m^2(2n-1)-2n-1}{6}
        \end{align*}{}
    \end{theorem}{}
    \begin{proof}
    We proceed similarly as we did for the Kirchhoff index but take the degrees of the vertices into account as Theorem \ref{thm:kemeny} calls for. If $u\in I$ then $d_u = 4$. Otherwise $d_u = 2$. 

The first summation term adds the effective resistance between vertices $u,v$ in $D_1$ in different copies of $G$ where exactly one of $u,v$ is in $I$. We take advantage of symmetry and multiply by $2$ to help accomplish this. The second term adds resistances where both $u,v\in I$. The third term adds resistance where $u,v\in D_1$ and  $u,v \in O$ and $u,v$ are in different copies of $G$.

The fourth term adds resistances with  $u,v \in D_2$, $u,v \in O$, and $u,v$ in different copies of $G$.

For all the following sums we will multiply by $2$ in order to count both $r(u,v)$ and $r(v,u)$.

The fifth term adds resistances where $u\in D_2$ and $v\in I$ and $u,v$ are in different copies of $G$. The sixth term adds resistances where $u,v\in D_2$ and $u,v \in O$ and $u,v$ are in different copies of $G$.

The seventh term adds resistances where $u,v\in D_1$, $v\in I$, and $u,v\in G_k$. The eighth term adds resistances where $u,v\in D_1$, $u,v\in O$, and $u,v\in G_k$.

The ninth term adds resistances where $u\in D_2$, $v\in I$, and $u,v\in G_k$. The tenth term adds resistances where $u\in D_2$, $v\in D_1$, $u,v\in O$, and $u,v\in G_k$.

The final term adds resistances where $u,v\in D_2$, $u,v \in O$, and $u,v \in G_k$.
    
    \begin{small}
    \begin{align*}
    \kemeny(F_n(C_m)) =& \frac{1}{4mn}\left[4\cdot 2\cdot2n\sum_{d=2}^n\sum_{k=1}^{p-1} \left(\frac{k(m-k)+p(m-p)(d-1)}{m}-\frac{(m-p)(k+p(d-1))^2}{nmp}\right)\right.\\
    &+4\cdot 4n\sum_{d=2}^n\frac{p(m-p)(n-d+1)(d-1)}{nm}\\
    &+2\cdot 2n\sum_{d=2}^n\sum_{l=1}^{p-1}\sum_{k=1}^{p-1} \left(\frac{k(m-k)+(p-l)(m-p+l)+p(m-p)(d-2)}{m}\right.\\
    &\hspace{12em}\left.-\frac{(m-p)(k-l+p(d-1))^2}{nmp}\right)\\
    &+2\cdot 2n\sum_{d=2}^n\sum_{l=1}^{m-p-1}\sum_{k=1}^{m-p-1}\left(\frac{(p+l)(m-p-l)+k(m-k)+p(m-p)(d-2)}{m}\right.\\
    &\hspace{12em}\left.-\frac{p(l-k+m+p(d-1)-md)^2}{nm(m-p)}\right)\\
    &+2\cdot4\cdot2n\sum_{d=2}^n\sum_{l=1}^{m-p-1} \frac{(p+l)(m-p-l)+p(m-p)(d-1)}{m}-\frac{p(l+d(p-m))^2}{nm(m-p)}\\
    &+2\cdot2\cdot2n\sum_{d=2}^n\sum_{l=1}^{m-p-1}\sum_{k=1}^{p-1} \left(\frac{(p+l)(m-p-l)+k(m-k)+p(m-p)(d-2)}{m}\right.\\
    &\hspace{12em}\left.-\frac{(p^2(d-1)+p(k+l+m-md)-km)^2}{nmp(m-p)}\right)\\
    &+2\cdot4\cdot2n\sum_{k=1}^{p-1} \left(\frac{k(m-k)}{m}-\frac{k^2(m-p)}{nmp}\right)\\
    &+2\cdot2\cdot2n\sum_{l=1}^{p-1}\sum_{k=l+1}^{p-1} \left(\frac{(k-l)(m-k+l)}{m}-\frac{(k-l)^2(m-p)}{nmp}\right)\\
    &+2\cdot4\cdot2n\sum_{l=1}^{m-p-1} \left(\frac{(p+l)(m-p-l)}{m}-\frac{p(m-p-l)^2}{nm(m-p)}\right)\\
    &+2\cdot2\cdot2n\sum_{l=1}^{m-p-1}\sum_{k=1}^{m-p-1} \left(\frac{(k+l)(m-k-l)}{m}-\frac{(p(k+l)-km)^2}{nmp(m-p)}\right)\\
    &+2\cdot2\cdot2n\sum_{l=1}^{m-p-1}\sum_{k=l+1}^{m-p-1} \left.\left(\frac{(k-l)(m-k+l)}{m}-\frac{p(k-l)^2}{nm(m-p)}\right)\right]
\end{align*}
\end{small}
    Simplifying these summations will yield the desired result.
    \end{proof}
    Comparing these results to the bounds from  Theorems \ref{thm:kirchbounds}, \ref{thm:kembounds} we find that as $n\to\infty$ the ratio of the upper bound for the Kirchhoff index to the actual Kirchhoff index approaches $\frac{3m^2}{(m-1)^2}$ and the ratio of the upper bound for Kemeny's constant to the actual Kemeny's constant approaches $3(m-1)^2$.

\bibliographystyle{plain}
\bibliography{sources}

\end{document}